\documentclass[11pt,a4paper,twoside]{article}

\usepackage[english]{babel}
\usepackage{amsfonts}
\usepackage[utf8]{inputenc}
\usepackage{graphicx}
\usepackage{setspace}
\usepackage[a4paper,top=33mm, bottom=27mm, left=20mm, right=20mm]{geometry}
\usepackage[affil-it]{authblk}
\usepackage{amsmath}
\usepackage[amsmath,amsthm,thmmarks]{ntheorem}
\usepackage{enumerate}
\usepackage[textsize=footnotesize,color=green!40]{todonotes}
\usepackage[colorlinks=false]{hyperref}

\theoremstyle{plain}
\newtheorem{thm}{Theorem}

\newtheorem{lemblank}{Lemma}

\newtheorem{lem}[thm]{Lemma}
\newtheorem{prop}[thm]{Proposition}
\newtheorem*{claim}{Claim}
\newtheorem{cor}[thm]{Corollary}

\theoremstyle{definition}
\newtheorem*{property}{Property $\mathbf{P(i)}$}

\newcommand{\E}[1]{{\mathbb E}\left[#1\right]}

\newcommand{\p}[1]{{\Pr}\left(#1\right)}

\newcommand{\kp}{{\rm P}}
\newcommand{\rem}{{\rm Q}}
\newcommand{\Eq}{{\rm Eq}}
\newcommand{\Vq}{{\rm Vq}}
\newcommand\expo{{\rm e}}

\newcommand\cC{{\cal C}}
\newcommand\cD{{\cal D}}
\newcommand\cE{{\cal E}}

\newcommand\ctn{{\Psi_1}}
\renewcommand\d{{\Delta}}

\newcommand\eps{{\varepsilon}}

\title{Acyclic edge colourings of graphs with large girth\footnote{\noindent Part of this work was done in April 2014 during the Workshop on Structural Graph Theory at McGill's Bellairs Institute. We warmly thank the organisers for the collaborative opportunity. The second and the third author also want to thank Ross Kang and Colin McDiarmid for fruitful discussions on this topic. Finally, the authors would like to thank the anonymous referees for carefully reading the manuscript and for the valuable comments provided.}}

\author{X. S. Cai\thanks{School of Computer Science, McGill University, 845 Sherbrooke Street West, Montreal, Quebec, Canada H3A 0G4. E-mail: xingshi.cai@mail.mcgill.ca},$\;$ G. Perarnau\thanks{School of Mathematics, University of Birmingham, Birmingham, B15 2TT, UK. E-mail: g.perarnau@bham.ac.uk},$\;$ B. Reed\thanks{CNRS, France; Kawarabayashi Large Graph ERATO Project, Japan. E-mail: breed@sophia.inria.fr}$\;$ and A. B. Watts\thanks{School of Computer Science, McGill University, 845 Sherbrooke Street West, Montreal, Quebec, Canada H3A 0G4. E-mail: adam.benwatts@gmail.com}}

\date{\today}

\begin{document}
\pagenumbering{arabic}

\setcounter{section}{0}

\maketitle

\onehalfspace

\begin{abstract}
An edge colouring of a graph $G$ is called acyclic if it is proper and every
cycle contains at least three colours. We show that for every $\varepsilon>0$,
there exists a $g=g(\varepsilon)$ such that if $G$ has maximum degree $\Delta$ and girth at least $g$ then
$G$ admits an acyclic edge colouring with $(1+\eps)\Delta{+O(1)}$ colours.
\end{abstract}
\section{Introduction}

\noindent An edge colouring of a graph $G=(V(G),E(G))$ is called \textit{acyclic} if it is proper (two edges that share an end point have different colours) and every cycle contains at least three colours (otherwise stated, there are no bicoloured cycles). The acyclic chromatic index of $G$, denoted by $a'(G)$, is the minimum number of colours used in an acyclic edge colouring of $G$. Acyclic colourings were introduced by Gr\"unbaum~\cite{grunbaum1973acyclic} in the context of vertex colouring of planar graphs (see~\cite{jensen2011graph} for more information on acyclic colourings).

Here we consider the problem of determining $a'(G)$ for a graph $G$ with bounded maximum degree. It is well known, that a proper edge colouring of a graph $G$ with maximum degree $\Delta$ may require $\Delta+1$ colours. Fiam{\v{c}}ik~\cite{fiamcik1978acyclic} and Alon, Sudakov and Zacks~\cite{alon2001acyclic} independently conjectured that $a'(G)\leq \Delta+2$. From an intuitive point of view, the conjecture states that at most one extra colour is needed to break all the bicoloured cycles of $G$. If true, the conjecture is best possible: there exists an infinite family of graphs $G$ for which $a'(G)= \Delta+2$~\cite{alon2001acyclic}.

Alon, McDiarmid and Reed~\cite{alon1991acyclic} showed that $a'(G)\leq 64 \Delta$. This result initiated a series of papers devoted to improve the constant. Molloy and Reed~\cite{molloy1998further} showed that $a'(G)\leq 16 \Delta$. This bound was improved by Ndreca, Procacci and Scoppola~\cite{ndreca2012improved} to $a'(G)\leq \lceil 9.62 (\Delta-1)\rceil$ using a new version of the L\'ovasz Local Lemma based on cluster expansion~\cite{bissacot2011improvement}. Esperet and Parreau~\cite{esperet2013acyclic} used the entropy compression method, based on the algorithmic version of the local lemma~\cite{moser2010constructive}, to show that $a'(G)\leq 4 \Delta-4$. 

The same problem has been studied in terms of the girth $g$ of $G$. Alon, Sudakov and Zacks~\cite{alon2001acyclic} showed that the acyclic edge colouring conjecture is true ($a'(G)\leq \Delta+2$) if $g\geq  c\Delta\log{\Delta}$, for some large constant $c$. They also noted that $a'(G)\leq 2\Delta+2$ if $g\geq  (1+o(1))\log{\Delta}$.
Unfortunately, the previous conditions imply that these results are only valid if the girth goes to infinity when $\Delta\to +\infty$. Weaker results have been obtained in the case when the girth is not too large~\cite{esperet2013acyclic,muthu2007improved,ndreca2012improved}. The best known result for large girth~\cite{Giota2014}, states that if $g\geq 219$, then $a'(G)\leq \lceil 2.323 (\Delta-1)\rceil+1$.

Recently, Bernshteyn~\cite{Bernshteyn2014} showed that for every $\eps>0$, $G$
has acyclic chromatic index at most $(2+\eps)\Delta$ provided that the girth of
$G$ is large only with respect to $\eps$. This result uses an Action
version of the Lov\'asz Local Lemma that refines the entropy compression argument~\cite{Berns2014,gonccalves2014entropy}.

Unfortunately, the standard entropy compression method seems to fail utterly if less than $2\Delta-1$ colours are used.
The heart of this method is a random recolouring procedure which recolours just one edge at a time.
When applying such a procedure, we can deduce no pseudorandom properties of the colourings
which we construct,
so we have to allow for the possibility that the $2\Delta-2$ edges adjacent to the edge that we
are recolouring all have different colours.

In this paper we improve the previous result of Bernshteyn by providing an asymptotically tight upper bound using an iterative colouring procedure.
\begin{thm}
    \label{thm:main}
    For every $\eps>0$, there exist constants $g=g(\eps)$ and
    $\Delta_0=\Delta_0(\eps)$, such that if $G$ has maximum degree
    $\Delta \geq \Delta_0$ and girth at least $g$, then
    $$
    a'(G)\leq (1+\eps)\Delta\;.
    $$
\end{thm}

The proof of this theorem uses the iterative edge colouring procedure introduced by Kahn to provide a list edge
colouring of a graph with $(1+o(1))\Delta$ colours~(see \cite{kahn1996asymptotically,molloy2000near}
or Chapter 14 in~\cite{Molloy2002graph}). Our main contribution is to track the partially coloured
cycles in $G$ that are still able to become bicoloured. In particular, we use the large
girth condition to make sure that at each iteration there are many uncoloured edges in any such cycle. Since the number of iterations of our procedure depends only on $\eps$, we need
the girth to be large only with respect to $\eps$.

We can restate the previous theorem in order that it holds for every $\Delta$.
\begin{cor}
For every $\eps>0$, there exist constants $C=C(\eps)$ and $g=g(\eps)$ such that
if $G$ has maximum degree $\Delta$ and girth at least $g$,
$$
    a'(G)\leq (1+\eps)\Delta+C\;.
$$
\end{cor}

From the proof of Theorem~\ref{thm:main}, one can easily derive an explicit lower bound on $g$
in terms of $\eps$; in particular, it suffices to set $g(\eps)= \varepsilon^{-O(1)}$.
We believe that this bound on $g$ in terms of $\eps$ can be improved but probably a new
approach is needed to get rid of the dependence {on} $\eps$ in the girth.

As far as we know, this is the first result on acyclic edge colourings of graphs with bounded girth that uses the asymptotically optimal number of colours. Finally, the proof of Theorem~\ref{thm:main} can be easily adapted to deal with acyclic list edge colourings.

\section{Reduction to regular graphs}\label{sec:regular}

\noindent {In this section we show that in order to prove Theorem~\ref{thm:main} it suffices to do it for $\Delta$-regular graphs. Although this is not a necessary step, it will simplify the proof of the main theorem.}
\begin{prop}\label{prop:regular}
If every $\Delta$-regular graph {with} girth at least $g$ admits an acyclic edge colouring with $a$ colours, then every graph with maximum degree $\Delta$ and girth at least $g$ also does.
\end{prop}
\begin{proof}

It is enough to show that we can embed $G$ into a $\Delta(G)$-regular
graph with the same girth. We do so now.

Suppose that $G_0=G$ is not regular. We will construct a graph $G_1$ by
considering many disjoint copies of $G_0$ and adding some edges between them. We will show that $G_1$ has a larger  minimum degree while it has the same maximum degree and girth.

The $r$-th power a graph $G$ is a graph which has the same vertex
set of $G$, but in which vertices are adjacent when their distance in $G$ is at
most $r$. Let $f_G$ be a colouring of the vertices of the $g$-th power of
$G_0$ using the set of colours $\{1,\dots,\Delta^g\}$. Since $G_0$ has maximum
degree $\Delta$, $\Delta^g$ colours suffice to obtain a proper colouring.

Let $H$ be a $\Delta^g$-regular bipartite graph with girth at least $g$. Since bipartite graphs are Class $I$ (see Lemma 1.4.18 in~\cite{lp1986}), we can obtain a $\Delta^g$ edge colouring $f_H$ of $H$; that is each of the $\Delta^g$ colours is incident to every vertex in $H$.

Now consider $G_1$ composed of $|V(H)|$ many disjoint copies of $G_0$ where
we add an edge between two copies of $v\in V(G_0)$ corresponding to
vertices $u_1,u_2\in V(H)$, if the degree of $v$ in $G_0$ is smaller than
$\Delta$, $u_1u_2\in E(H)$ and $f_H(u_1u_2)=f_G(v)$. The copy of $G_0$ in
$G_1$ corresponding to $u\in V(H)$ will be denoted by $G_0^{u}$.

Observe that the graph $G_1$ has the same maximum degree
as $G_0$ because we only add edges between vertices that have a degree smaller
than $\Delta$. Also the minimum degree has increased by $1$ since for every
vertex $u\in V(H)$ and every colour $c\in\{1,\dots,\Delta^g\}$ there is a
vertex $u'\in V(H)$ such that $uu'\in E(H)$ and $f_H(uu')=c$. Thus, every copy
of a vertex that has degree smaller than $\Delta$ in $G_0$ is incident to a new
edge in $G_1$.

Moreover, $G_1$ has girth at least $g$. Suppose that there is a cycle $C$ of
length $\ell$ which is strictly less than $g$. Since $G_0$ has no cycles of length
$\ell$, the cycle should contain vertices in different copies of $G_0$. In
particular, the cycle induces a closed walk in $H$ of length at least $2$
and at most $\ell$. If this walk contains a cycle, let $C'$ be one of the
minimal ones. Clearly the length of $C'$ is at most $\ell$, but $H$ has
girth at least $g$, a contradiction. Thus, we may assume that the walk is
acyclic. Let $u\in V(H)$ be one of the leaves of the acyclic walk (there at
least two of them) and let $u'\in V(H)$ be its unique neighbour in the
closed walk. Otherwise stated, the cycle at some point enters in a vertex
$v_1\in V(G_0^{u})$ from $G_0^{u'}$, stays in $G_0^{u}$ for a while and
exits from  $v_2\in V(G_0^{u})$ towards $G_0^{u'}$. Since $C$ has length
$\ell$, there exists a path in $G_0^{u}$ from $v_1$ to $v_2$ of length at
most $\ell$. Moreover, since $v_1$ and $v_2$ have edges to $G_0^{u'}$,
$f_G(v_1)=f_H(uu')=f_G(v_2)$. Hence, there are two vertices from $G_0$ at
distance less than $g$ with the same colour, obtaining a contradiction with the
construction of $f_G$.

Repeating the same argument at most $\Delta$ times, we embed the graph $G$ in a
$\Delta$-regular graph $G_\Delta$ with girth at least $g$.
\end{proof}

\section{Outline of the proof of Theorem~\ref{thm:main}}\label{sec:outline}

\noindent In this section we sketch the proof of Theorem~\ref{thm:main}.

Throughout the proof, we will assume that $\eps\leq \eps_0=10^{-3}$. Observe that if we show
Theorem~\ref{thm:main} for any $\eps\leq \eps_0$, then it also holds for all $\eps>0$ by setting
$g(\eps)=g(\eps_0)$ and $\Delta_0(\eps)=\Delta_0(\eps_0)$ if $\eps\geq \eps_0$. For the sake of
clarity of the presentation, we will omit floors and ceilings whenever they are not relevant for the argument.
We use the notations \(O(\cdot), o(\cdot)\), and \(\Omega(\cdot)\) for classes of functions of
\(\Delta\) in the usual sense. We also use \(O(1)\) to denote an implicit constant that does not
depend on \(\Delta\) but could depend on \(\varepsilon\). Unless explicitly stated otherwise, all
logarithms are natural. We use \(V(G)\) and $E(G)$ to denote the vertex and the edge set of \(G\). Finally, for every vertex $v$, we use $N(v)$ to denote the set of neighbours
of $v$.

Our proof is an analogue of the one for list chromatic index in Chapter~14
of~\cite{Molloy2002graph} where at the same time we control how many edges are
coloured in partially bicoloured cycles. We will use the same notations as
in~\cite{Molloy2002graph} and mimic some of the arguments displayed
there. Recall that by Proposition~\ref{prop:regular}, we can assume that the graph $G$ is $\Delta$-regular. This allows us to simplify some technicalities in the proof of Theorem~\ref{thm:main}.

We will prove the following three lemmas from which the
desired result follows.
The first lemma says that {at the beginning} we can reserve a small number of colours for each
vertex, satisfying some useful properties. We are going to use these colours at
the end.
\begin{lemblank}
    \label{lem:pre_reservation}
There exists a collection $\{S_v:\,v\in V(G)\}$ of subsets of $[(1+\eps)\Delta]$ such that
\begin{enumerate}[({A}.1)]
\item  for every vertex $v$, $|S_v| \leq \frac{4\eps\Delta}{9}$,
\item  for every edge $e=uv$, $|S_u\cap S_v| \geq \frac{\eps^2 \Delta}{18}$, and
\item  for every vertex $v$ and every colour $c$, $|\{u\in N(v):\; c\in S_u\}| \leq \frac{\eps\Delta}{2}$.
\end{enumerate}
\end{lemblank}

The second lemma is the core of the proof of Theorem~\ref{thm:main} and shows that, given a certain collection of sets of reserved colours on the vertices of $G$ (for instance, the one provided by Lemma~\ref{lem:pre_reservation}), there exists a partial edge colouring that satisfies some desirable properties.
\begin{lemblank}
    \label{lem:pre_colouring}
Let $\{S_v:\,v\in V(G)\}$ be a collection of subsets of $[(1+\eps)\Delta]$ that satisfies properties (A.1)--(A.3) from Lemma~\ref{lem:pre_reservation}.
Then, there exists a proper partial acyclic edge colouring of $G$ such that
\begin{enumerate}[({B}.1)]
\item every coloured edge $e=uv$ has a colour from $[(1+\eps)\Delta]\setminus (S_u\cup S_v)$;
\item for every vertex $v$ and every colour $c\in S_v$ there are at most
    $\left( \frac{\varepsilon^2}{18} \right)^2 \frac{\Delta}{128}$ vertices $u \in N(v)$ satisfying that $uv$ is
        uncoloured and
    $c \in S_u$, and
\item every cycle $C$ that can become completely bicoloured by using on the uncoloured edges $e=uv$ of $C$
a colour in \(S_u \cap S_v\),
has at least $3$ uncoloured edges.
\end{enumerate}
\end{lemblank}

The collection of sets provided by Lemma~\ref{lem:pre_reservation} can be used in
Lemma~\ref{lem:pre_colouring} to obtain a partial  acyclic edge colouring satisfying some nice
properties. The following lemma shows that, using these properties, we can complete the partial acyclic edge
colouring provided by Lemma~\ref{lem:pre_colouring}.
\begin{lemblank}
    \label{lem:pre_finishing}
    Let \(\gamma \in (0,1)\) be a given constant. Let \(\chi\) be a partial acyclic edge colouring of \(G\) that satisfies the following properties,
    \begin{itemize}
    \item[-] for every uncoloured edge \(e\) there is a list of colours of size exactly \(\gamma \Delta\) that are not used by $\chi$ in the coloured edges adjacent to $e$,
    \item[-] for every uncoloured edge \(e\), every colour in its list appears in the list of at most \(\gamma^2 \Delta/128\) uncoloured edges adjacent to \(e\),
    \item[-] there are at most \((1+\varepsilon)\Delta\) colours in the union of all the lists, and
    \item[-] every cycle that can be completely bicoloured in an extension of \(\chi\), has at least \(3\) uncoloured edges.
    \end{itemize}
    Then, \(\chi\) can be extended to an acyclic edge colouring of $G$.
\end{lemblank}

Lemma~\ref{lem:pre_reservation} and {Lemma}~\ref{lem:pre_finishing} are relatively straightforward
applications of the Local Lemma. Their proof is given in Section \ref{sec:reservation} and
\ref{sec:finish} respectively. We dedicate the rest of this section to give an outline
of the proof of Lemma~\ref{lem:pre_colouring}.

We prove the existence of the desired colouring via an iterative procedure with $i^*$ iterations,
where $i^*$ only depends on $\eps$ and will be defined later.
For every $1\leq i\leq i^*$, at the beginning of the $i$-th iteration, we will have a partial edge colouring of $G$. Moreover, for every uncoloured edge, we will also have a list of available
colours for it. Then, we will perform the following steps in order to obtain a new partial edge colouring:
\begin{itemize}
    \item[S.1:] Each list of colours is truncated to the same size by removing some
        colours in an arbitrary way.
    \item[S.2:] Each uncoloured edge is assigned a colour selected uniformly at random
        from its list.
    \item[S.3:] Adjacent edges assigned the same colour are uncoloured.
    \item[S.4:] For each edge with a newly assigned colour, an equalizing coin is flipped to
        decide if it gets uncoloured.
    \item[S.5:] Colours retained on an edge are removed from the lists of the edges adjacent to it.
    \item[S.6:] For every colour and vertex, an equalizing coin is flipped to decide if the colour is
        removed from all the lists of the edges incident to the vertex.
\end{itemize}

After performing these steps, we will show that there exist constants $L_{i+1}, T_{i+1}, R_{i+1}$ and $\Lambda_{i+1}$ (to be
set later) such that the partial edge colouring that we have obtained, satisfies the following
properties:
\begin{itemize}
    \item[--] for every uncoloured edge, there are at least $L_{i+1}$  choices of colours {in its corresponding list},
    \item[--] for every vertex $v$ and every colour $c$, there are at most $T_{i+1}$
         edges $e=uv$ satisfying that $e$ is uncoloured
         and \(c\) is in the list of \(e\),
    \item[--] for every vertex $v$ and every colour $c$, there are at most $R_{i+1}$
         vertices $u \in N(v)$ satisfying that $uv$ is uncoloured
        and $c\in S_u$, and
    \item[--] every cycle that is partially bicoloured and that is
            \textit{significant} (to be defined later) has at least $\Lambda_{i+1}$ uncoloured \textit{free} (to be defined later) edges.
\end{itemize}

Now, we define the quantities related to the partial edge colouring that we need to control throughout the iterative colouring.
For every edge $e=uv$, let $L_1(e) := [(1+\eps)\Delta] \setminus (S_u \cup S_v)$, be the initial
list of available colours.

For every uncoloured edge $e$, let $L_{i}(e)$ be the list of available colours in $e$ {\it before} the $i$-th iteration.
(For the sake of clarity and with a slight abuse of notation, we also use \(L_{i}(e)\) to refer to the list of available colours at edge $e$ after the truncation at step S.1.)
Let $\ell_{i}(e) := |L_{i}(e)|$.

For every vertex $v$ and every colour $c$, let $T_{i}(v,c)$ be the set {of edges $e$ incident to $v$} such that the following holds: {\it before} the $i$-th iteration $e$ is uncoloured and {$c\in L_i(e)$}.
Let $t_{i}(v,c) := |T_{i}(v,c)|$.

Finally, for every vertex $v$ and every colour $c$, let $R_{i}(v,c)$ be the set {of} neighbours $u$ of $v$ such that the following holds: {\it before} the $i$-th iteration {$uv$} is uncoloured {and $c\in S_u$}. Let $r_{i}(v,c) := |R_{i}(v,c)|$.

Observe that $\ell_{i}(e)$, $t_{i}(v,c)$ and $r_{i}(v,c)$ control the first three quantities which we
need to bound at every iteration. We define $L_{i}$, $T_{i}$ and $R_i$ recursively as follows.  Let
$ L_1 := \left(  1+\eps/9 \right) \Delta$, $T_{1} := \Delta$ and $R_{1} := \eps \Delta/2$.  Then, for every $i\geq 1$
\begin{align}\label{eq:definition}
    L_{i+1} &:= (1-\expo^{-2})^{2} \,L_{i} - L^{2/3}_{i}, \nonumber\\
    T_{i+1} &:= (1-\expo^{-2})^{2} \,T_{i} + T^{2/3}_{i}, \\
    R_{i+1} &:= (1-\expo^{-2}) R_{i} + R^{2/3}_{i}.\nonumber
\end{align}


Let us now focus on the last property that the partial edge colouring must satisfy at the end of the $i$-th iteration.
First of all, we give some definitions.

At the beginning of the \(i\)-th iteration, \(G\) has a partial edge colouring and each
uncoloured edge $e=uv$ has two lists of colours: $L_i(e)$ (the colours that can be used for $e$ during the \(i\)-th iteration) and $S_u\cap S_v$ (the colours that have been reserved for the end).

For every colour \(c\) and every uncoloured edge $e=uv$, we say that $e$ is $c$-\textit{reserved} if $c\in S_u\cap S_v$ and that $e$ is $c$-\textit{free} if $c\in L_i(e)$. Given a pair of colours \(\{c,d\}\), we also use the term $\{c,d\}$-reserved to refer to edges that are either $c$-reserved, $d$-reserved, or both (and similarly for the free edges). By construction of $L_i(e)$, an edge cannot be both reserved and free.
We say that an edge $e=uv$ is \(c\)-\textit{compatible} if it has colour $c$ or it is uncoloured and either $c$-reserved or $c$-free.  We say that a cycle \(C\) is \(\{c,d\}\)-\textit{compatible} if every edge
in \(C\) is either \(c\)-compatible, \(d\)-compatible, or both.

Given a path \(v_1, \ldots, v_{\ell}\) and an ordered pair of colours $(c,d)$ we call the path
$(c,d)$-alternating if, for every $j\geq 1$, \(e = v_{2j-1} v_{2j}\) is a \(c\)-compatible edge and
\(e = v_{2j} v_{2j+1}\) is a $d$-compatible edge.
For a $\{c,d\}$-compatible cycle $C$, we define its $\{c,d\}$-multiplicity as the smallest $t$ such that
$C$ can be partitioned into $t$ paths that are either $(c,d)$-alternating or $(d,c)$-alternating.  We define the $\{c,d\}$-multiplicity of
a path in the same way.

%

In order to control the multiplicity and the number of reserved edges in a cycle, we define
$$
\Psi_{i} := 4^{2+i^{*}-i}\;.
$$
We say that a cycle \(C\) is \( \{c,d\}\)-\textit{significant} at
the \(i\)-th iteration, if at the beginning of
the $i$-th iteration, \(C\) is \( \{c,d\}\)-compatible, has \(\{c,d\}\)-multiplicity at most \(\Psi_{i}\)  and contains at most \(\Psi_{i}\)  \(
\{c,d\}\)-reserved edges. (When the pair of colours $\{c,d\}$ is clear from the context, we simply refer to compatible and significant cycles, to multiplicity, and to reserved and free edges.)

Note that since \(\Psi_i\) is decreasing in \(i\), the definition of significant cycles becomes stronger throughout the iterative procedure.
Thus, the number of significant cycles that need to be controlled is considerably reduced at each iteration, which is a crucial point for our analysis.


Let $k:=\lfloor g/2 \rfloor$, where $g$ is the girth of $G$. For every pair of colours \( \{c,d\}\) and every cycle $C$ that is \( \{c,d\}\)-significant at the \(i\)-th
iteration, let $\lambda_{i}^{\{c,d\}}(C)$ be its number of \( \{c,d\}\)-free
edges at the beginning of the $i$-th iteration. In order to lower bound $\lambda_{i}^{\{c,d\}}(C)$, we define
$$
\Lambda_i := \frac{2k}{2^{i-1}} - 4 \ctn i\;.
$$
With this definition, we have $\Lambda_{i+1}=2\Lambda_i + 4\Psi_1$.

Now, we are able to precisely state the properties of the partially edge-coloured graph that our
colouring procedure must satisfy at every iteration.
For every $0 \leq i \leq i^*$, we define:
\begin{property} {With the definitions given above, the following is satisfied:}
\begin{enumerate}[(P.1)]
    \item $\ell_{i+1}(e) \ge L_{i+1},$ for every uncoloured edge $e$,
    \item $t_{i+1}(v,c) \le T_{i+1},$ for every vertex $v$ and colour $c$,
    \item $r_{i+1}(v,c) \le R_{i+1},$ for every vertex $v$ and colour $c$, and
    \item {$\lambda^{\{c,d\}}_{i+1}(C) \ge \Lambda_{i+1}$, for every pair of colours \( \{c,d\}\) and
        every cycle $C$ that is \( \{c,d\}\)-significant at the \( (i+1)\)-st iteration.}
\end{enumerate}
\end{property}

\noindent Let us now define the total number of iterations $i^{*}=i^{*}(\eps)$ as the smallest integer $i$ such that
\begin{align}\label{eq:cond_on_i*}
   R_{i+1}< \left(\frac{\varepsilon^2}{18}\right)^2 \frac{\Delta}{128} \;.
\end{align}
We stress that if $\Delta$ is large enough, then $i^*$ only depends on $\eps$, since $R_1=\eps\Delta/2$ and $R_{i+1}= (1-\expo^{-2})
R_{i} + R^{2/3}_{i}$. In particular, we have $i^*\leq c \log_2{(1/\eps)}$, for some constant
$c>0$.

\begin{lem}\label{lem:finalP}
    If property $P(i^*)$ is satisfied, then there exists a proper partial acyclic edge colouring of $G$ that satisfies the properties (B.1)--(B.3) of Lemma~\ref{lem:pre_colouring}.
\end{lem}
\begin{proof}
First of all, since every edge $e$ that is coloured during the iterative procedure, is assigned a colour from $L_1(e)=[(1+\eps)\Delta]\setminus (S_u\cup S_v)$, property (B.1) is satisfied.
Moreover, by the definition of the stopping time $i^*$ in~\eqref{eq:cond_on_i*}, condition (B.2) is also satisfied.

Finally, we will show that $\lambda^{\{c,d\}}_{i^*+1}(C) \ge \Lambda_{i^*+1}=2k/2^{i^*}-4\Psi_1 (i^*+1)$ implies condition (B.3).
Since $i^*$ and $\Psi_1$ only depend on $\eps$, we can consider $g=g(\eps)$ (and thus, also $k$) large enough with respect to $\eps$ such that $\Lambda_{i^{*}+1}\geq {3}$. {This particularly implies that $k\approx 2^{i^*} = O\left(\eps^{-c}\right)$.}
Let $C$ be a partially edge-coloured cycle that can become completely bicoloured with colours $c$ and $d$, by using in each uncoloured edge $e=uv$, colours from $S_u\cap S_v$.
If $C$ is significant at the $(i^*+1)$-st iteration, then there are at least $\Lambda_{i^{*}+1}\geq {3}$ free (and in particular, uncoloured) edges.
Therefore assume that $C$ is not significant. Since $C$ can become bicoloured with colours $c$ and
$d$, it must be $\{c,d\}$-compatible and must have $\{c,d\}$-multiplicity $1$. Thus, the only way that $C$ fails to be significant is in the case that there are more than \(\Psi_{i^{*}+1} = 4^{2+i^{*}-(i^{*}+1)}=4\) reserved (and in particular, uncoloured) edges.
In conclusion, any such cycle has at least ${3}$ uncoloured edges and property (B.3) also holds.
\end{proof}

In order to prove that {$P(i^*)$} holds, we
first show that $P(0)$
holds by the hypothesis of Lemma~\ref{lem:pre_colouring}.
\begin{lem}\label{lem:P(0)}
Property $P(0)$ is satisfied deterministically.
\end{lem}
\begin{proof}
For every edge $e$, we have $\ell_1(e)\geq (1+\eps)\Delta-2\cdot \frac{4\eps \Delta}{9} = L_1$ by the property
(A.1). For every vertex $v$ and every colour $c$, we have $t_1(v,c)\leq \Delta=T_1$. For every
vertex $v$ and every colour $c$, we have $r_1(v,c)\leq \frac{\eps \Delta}{2} = R_1$ by the property (A.3). Finally,
since for every pair of colours $\{c,d\}$, every $\{c,d\}$-significant cycle $C$ contains at least
\(g\geq 2k\) uncoloured edges and at most \(\Psi_1\) reserved edges (by the definition of significant
cycles), \(C\) has at least \(2k-\Psi_1 > 2k- 4 \Psi_1 = \Lambda_1\) free edges.
\end{proof}

Thus, it suffices to show that for every $1\leq i\leq i^*$, if property $P(i-1)$ holds, then with positive probability $P(i)$ also holds.
In the remainder of this section, we provide an intuition on why this is true. {We will do it by showing that the expected decrease of the quantities $\ell_{i}(e)$, $t_{i}(v,c)$ and $r_{i}(v,c)$, is governed by the parameters $L_i$, $T_i$ and $R_i$, respectively.}
In Section~\ref{sec:main_ite} we present the details of it.

Observe that after the truncation at step S.1, we have $\ell_{i}(e) = L_i$ for every edge $e$. Assume that for every vertex $v$ and every colour $c$, we have $t_{i}(v,c) = T_i$ and $r_{i}(v,c) = R_i$. This assumption will help us to derive an intuition on how the previous parameters drop at each iteration. The equalizing coins will be used to correct the possible fluctuations of the previous parameters. Since we are assuming that these parameters have a given value, it also makes sense to assume that the equalizing coins are not flipped (that is, steps S.4 and S.6 are skipped in this intuitive analysis).

For every vertex $v$ and every colour $c$, let $\rem_i(v,c)$ be the probability that exactly one edge incident to $v$  retains
colour $c$ {assigned at step S.2}. Observe that, by step S.5 in our procedure, this is precisely the same as the probability that $c$ is removed from the lists of all the edges incident to $v$ because an edge incident to $v$ retains it. Since after step S.1 we have $\ell_i(e)=L_i$,
$$
{\rem_{i}(v,c)}
{=} \frac {T_i} {L_i} \left( 1 - \frac{1}{L_i} \right)^{T_i - 1} \left( 1 -
\frac{1}{L_i} \right)^{T_i-1}
= \frac {T_i} {L_i} \left( 1 - \frac{1}{L_i} \right)^{2 T_i-2} \approx
 \expo^{-2},
$$
{where the last approximation follows since $L_i$ and $T_i$ are very close at the beginning of the iterative colouring~(see definitions of $L_1$ and $T_1$) and they drop at the same speed, as it will be explained below. Since this probability does not depend on the choice of $v$ and $c$, we write $\rem_{i}:=\rem_{i}(v,c)$.}

For $e=uv$, we can compute the expected value of $\ell_{i+1}(e)$,
$$
\E{ \ell_{i+1}(e)}= {\E{\sum_{c\in L_i(e)} (1-\rem_{i}(u,c))(1-\rem_{i}(v,c))}} {=} L_{i} (1-\rem_i)^{2} \approx L_i (1-\expo^{-2})^{2} \approx L_{i+1},
$$
where the last approximation follows from \eqref{eq:definition}.

For every edge $e$ and every colour $c$, let $\kp_i(e,c)$ be the probability that every edge adjacent to $e$ is not assigned colour $c$ at step S.2. Then
$$
{\kp_i(e,c)} {=} \left( 1- \frac{1}{L_i} \right)^{2(T_i - 1)} \approx \expo^{-2}.
$$
Similarly as before, since the previous probability does not depend on the choice of $e$ and $c$, we write $\kp_{i}:=\kp_{i}(e,c)$. Observe that the probability that an edge $e$ retains its colour is exactly $\kp_{i}$: let $c$ be the colour assigned to $e$ at step S.2, then the probability it is retained at step S.3, is the probability $c$ is not assigned to any of the edges adjacent to $e$.

For every vertex $v$ and every colour $c$, there are two ways that an edge $e=vw\in T_i(v,c)$ satisfies $e\notin T_{i+1}(v,c)$: (1) either $c$ is removed from the list corresponding to $e$ at step S.5 because it was assigned to an edge incident to $w$~(which happens with probability $\rem_i$) or (2) $e$ retains the colour assigned at step S.2~(which happens with probability $\kp_i$). Observe that if $c$ is removed from the list corresponding to $e$ because it was assigned to an edge incident to $v$, then, for every edge $f$ incident to $v$ we have $c\notin L_{i+1}(f)$ and we can stop tracking $T_{i+1}(v,c)$.
Since the probability of (1) and (2) are essentially independent, it follows from the above estimations and by~\eqref{eq:definition} that,
$$
\E{t_{i+1}(v,c)} \approx T_{i} (1-\kp_i)(1-\rem_i) \approx T_{i}(1- \expo^{-2})^2 \approx T_{i+1}.
$$
Thus, as we stated before, roughly speaking, $\ell_{i}(e)$ and $t_{i}(v,c)$ drop at the same speed.

For every vertex $v$ and every colour $c$, the only way that a vertex $u\in R_i(v,c)$ satisfies $u\notin R_{i+1}(v,c)$ is that the edge $uv$ retains its assigned colour which occurs with probability $\kp_i$. Using~\eqref{eq:definition} again,
$$
\E{r_{i+1}(v,c)} {=} R_{i} (1-\kp_i) \approx R_{i}(1- \expo^{-2}) \approx R_{i+1}.
$$
It is worth noticing that, on the one hand, the expected drop of $r_i(v,c)$ is smaller than the one of $\ell_{i}(e)$, but, on the other hand, $r_1(v,c)\leq R_1=\eps\Delta/2$ is much smaller than $\ell_{1}(e)\geq L_1=(1+\eps/9)\Delta$.

The control of the number of free edges in significant cycles is more involved and will be done directly in Section~\ref{sec:main_ite}.




\section{Proof of Lemma~\ref{lem:pre_reservation}}

\label{sec:reservation}

{In this section we prove that there exists a collection of subsets of $[(1+\eps)\Delta]$ that satisfies some desirable properties.}
\begin{proof}[Proof of Lemma \ref{lem:pre_reservation}]
For every vertex $v\in V(G)$ we construct the set $S_v$ by selecting each
colour $c$ independently with probability $(1+\eps)^{-1/2}\eps/3.$
{We show that with positive probability this collection of sets fulfils the conditions (A.1)--(A.3) in Lemma
\ref{lem:pre_reservation}.}

Recall that for every edge $e=uv$, $S_e=S_v\cap S_u$.
Let $S_{v,c}:=\{u\in N(v):\; c\in S_u\}$, be the set of neighbours of $v$ that reserve colour $c$. Note that $S_{v,c}= R_1(v,c)$, with the definition of $R_i(v,c)$ given in Section~\ref{sec:outline}.
For every vertex $v$, edge $e$ and colour $c$
$$
\E{|S_v|} = \frac{\eps(1+\eps)^{1/2}\Delta}{3}\;,   \quad  \E{|S_e|} = \frac{\eps^2\Delta}{9} \quad \mbox{ and } \E{|S_{v,c}|} = \frac{\eps\Delta}{3(1+\eps)^{1/2}} \;.
$$

For every vertex $v$ let $A_v$ be the event that $|S_v| >
\frac{4\eps\Delta}{9}$, for every edge $e$ let $B_e$ be the event that $|S_e| <
\frac{\eps^2\Delta}{18}$ and for every vertex $v$ and colour $c$ let $C_{v,c}$
be the event that $|S_{v,c}|> \frac{\eps\Delta}{2}$.

Observe that the random variables $|S_v|$, $|S_e|$ and $|S_{v,c}|$ are binomially distributed.
We use Chernoff's inequality~\cite[chap.~5]{Molloy2002graph}, which states that if $X$ is a binomial with $N$ trials and probability $p$, for any $\delta\in (0,1)$:
$$
\p{X\leq (1-\delta)\E{X}} <e^{-\frac{\delta^2}{2}\E{X}}\;,
$$
and
$$
\p{X\geq (1+\delta)\E{X}} <e^{-\frac{\delta^2}{3}\E{X}}\;.
$$

Recall that $\eps\leq \eps_0=10^{-3}$. Then, we have
\begin{align*}
\p{A_v}& \leq \p{|S_v|> (1+1/4)\E{|S_v|}}\leq \expo^{-(1/4)^2\,\E{|S_v|}/3} = \expo^{-\Omega(\Delta)} \\
\p{B_e}& =  \p{|S_e|< (1-1/2)\E{|S_e|}}\leq \expo^{-(1/2)^2\,\E{|S_e|}/2} = \expo^{-\Omega(\Delta)} \\
\p{C_{v,c}}& \leq  \p{|S_{v,c}| >  (1+1/2)\E{|S_{v,c}|}} \leq \expo^{-(1/2)^2\, \E{|S_{v,c}|}/3}= \expo^{-\Omega(\Delta)} \;.
\end{align*}

Observe that all the events $A_v$ are mutually independent. An event $A_v$ depends
on at most $\Delta$ events of type $B$, precisely the ones for which $e=uv$, and
on at most $(1+\eps)\Delta^2$ of type $C$, one for each vertex $u\in N(v)$ and
colour $c$. Analogously, an event $B_e$ depends on at most $2$ events of type
$A$, on at most $2(\Delta-1)$ {other} events of type $B$, and on at most
$2(1+\eps)\Delta^2$ events of type $C$. Finally each event $C_{v,c}$
depends on at most $\Delta$ events of type $A$, on at most $\Delta^2$ events of
type $B$ and on at most $\Delta^2$ {other} events of type $C$.

The Lov\'asz Local Lemma {(in the form of~\cite[p.~221]{Molloy2002graph})} directly shows that if $\Delta$ is large enough with
respect to $\eps$, with positive probability all the events do not hold at the
same time, thus proving the lemma.
\end{proof}

\section{Proof of Lemma~\ref{lem:pre_colouring}}\label{sec:main_ite}

{In this section we prove the lemma where most of the acyclic edge colouring is constructed.} First we compute the precise expectation of each
random variable involved in the iterative colouring procedure. Then, we show that these random variables are
concentrated. Finally we prove that with positive probability at the end of each iteration the graph has a partial edge colouring that satisfies the desired properties.

\subsection{The Expected Values}\label{sec:exp}

{Recall all the definitions given in Section~\ref{sec:outline}.} In this subsection we compute the expected values of $\ell_{i+1}(e)$,
$t_{i+1}(v,c)$ and $r_{i+1}(v,c)$, given that property $P(i-1)$ holds. For simplicity, let $\eta := (1-\expo^{-2})^2$ throughout this section.

The first lemma in this section shows that $L_i$, $T_i$, and $R_i$ do not become too small before
the {$i^{*}$-th} iteration.
\begin{lem}\label{lem:i:star}
    {There {exist positive} constants
    $\varepsilon_1$, $\varepsilon_2$, and $\varepsilon_3$ (only depending on $\eps$) such that for
    every $1\leq i \le i^{*}+1$},
    $$
    L_{i} \ge \varepsilon_1 \Delta, \qquad T_{i} \ge \varepsilon_2 \Delta,
    \qquad R_{i} \ge \varepsilon_3 \Delta.
    $$
\end{lem}
\begin{proof}
    {Recall that $i^{*}$ is defined in~\eqref{eq:cond_on_i*} as the smallest integer such
    that $R_{i+1}< \left(\frac{\varepsilon^2}{18}\right)^2 \frac{\Delta}{128}= \frac{\varepsilon^3}{18^2}\cdot \frac{R_1}{128}$. Besides, by~\eqref{eq:definition} we also have $R_{i^*+1}\geq \eta^{i^{*}/2}R_1$. Therefore,
    $
    \eta^{i^{*}} \ge \left( \frac{\varepsilon^3}{18^2\cdot 128} \right)^{2}.
    $
    Let $\eps_1=\eps^8$. For every $i \le i^{*}+1$, it follows that 
    \begin{align*}
    L_{i} &\ge  \eta^{i^{*}}L_{1}-i^* L_1^{2/3}
    \ge (1-o(1))\left(\frac{\varepsilon^3}{18^2\cdot 128}\right)^{2}\left(1+\frac \varepsilon 9\right) \Delta\geq \eps_1\Delta\;,
    \end{align*}
    where we used that $\eps\leq \eps_0=10^{-3}$. Similar arguments can be applied for $T_i$ and $R_i$.}
\end{proof}

Recall the definitions of $L_1$, $T_1$ and $R_1$ and define the following three sequences for every $i\geq 0$,
\begin{align*}
L_{i+1}' := \eta^{i} L_1, \qquad
T_{i+1}' := \eta^{i} T_1, \qquad
R_{i+1}' := \eta^{i/2} R_1.
\end{align*}
The next lemma shows that $L_i$, $T_i$ and $R_i$, {as defined in~\eqref{eq:definition}}, are very close to $L_i'$,
$T_i'$ and $R_i'$ respectively.
\begin{lem}\label{lem:no:erro}
    For every $1\leq i \le i^{*}+1$, we have\\
    (a) $|L_i - L_i'| \le (L_i')^{5/6} = o(L_{i}')$, \\
    (b) $|T_i - T_i'| \le (T_i')^{5/6} = o(T_{i}')$, \\
    (c) $|R_i - R_i'| \le (R_i')^{5/6} = o(R_{i}')$.
\end{lem}

\begin{proof}
    {Let us prove part (a) using induction on $i$. The proofs of parts (b) and (c) use identical arguments and are omitted.}

    The base case $i = 1$ is
    clearly true as $L_1'=L_1$. If (a) holds for $i$,
    then
    \begin{align*}
        |L_{i+1}' - L_{i+1}| & = L_{i+1}' - L_{i+1} = \eta L_{i}' - \eta L_{i} + L_{i}^{2/3}  \le \eta (L_{i}')^{5/6} + L_{i}^{2/3},
    \end{align*}
    where the first equality comes from the obvious fact that $L_j \le L_j'$
    for all $j$.
    To see that the expression above  is at most $(L_{i+1}')^{5/6}$, it suffices to note that
    \begin{align*}
        (L_{i+1}')^{5/6} - \eta (L_{i}')^{5/6}
        = (\eta^{5/6} - \eta) (L_{i}')^{5/6}
        \ge (L_{i}')^{2/3}
        \ge L_{i}^{2/3},
    \end{align*}
    {where in the first inequality, we used that $L_i'\geq L_i\geq \eps_1\Delta$ from Lemma~\ref{lem:i:star}.}
\end{proof}

This other lemma is a consequence of Lemma~\ref{lem:no:erro} and shows that $L_i$ and $T_i$ are close throughout the iterative colouring.
\begin{lem}\label{lem:L:T:close}
For every $1\leq i \le i^{*}+1$, we have
    ${L_i/T_i = 1 + \varepsilon/9+o(1)}$.
\end{lem}

\begin{proof}
    Since $L_i \le L_{i}'$ and $T_i \ge T_{i}'$,
    $$
    \frac {L_i} {T_i}
    \le \frac {L_i'} {T_i'}
    = \frac {L_1 \eta^{i-1}} {T_1 \eta^{i-1}}
    =  1 + \frac \varepsilon 9 .
    $$
    The other direction follows from Lemma~\ref{lem:no:erro},
    \begin{equation*}
        \frac {L_i} {T_i}
        \ge \frac {(1-o(1))L_i'} {(1+o(1))T_i'}
        =   1 + \frac \varepsilon 9 -o(1).
    \end{equation*}
\end{proof}

Recall that $L_{i}(e)$ is the list of colours that are still
available to $e$ at the beginning of the $i$-th iteration.  For a vertex $v$,
we define
$$
L_{i}(v) := \bigcup_{\substack{e \ni v: \\ \text{{$e$ uncoloured}}}} L_{i}(e)\;,
$$
to be the union of the colour lists
of all {uncoloured} edges incident to $v$.
When we say that a colour $c$ is removed from the list $L_i(v)$, it is also removed from every list $L_i(e)$ with $e$ {uncoloured and} incident to $v$.

In Section~\ref{sec:outline} we sketched the steps that we perform at the $i$-th iteration of the
procedure. Here we make it precise, provided that $P(i-1)$ is satisfied:
\begin{enumerate}[{S}.1:]
    \item For every {uncoloured} edge $e = uv$, truncate $L_i(e)$ {by
        removing colours in an arbitrary way until it has size precisely $L_{i}$ (recall that by (P.1), we have $\ell_i(e)\geq L_i$).} When a colour $c$
        is removed from $L_i(e)$, $e$ is also removed from $T_i(v,c)$ and
        $T_i(u,c)$.
    \item For every uncoloured edge $e$, assign it a colour chosen uniformly at random from
        $L_i(e)$.
    \item Uncolour every edge $e$ which is assigned the same colour as one of its adjacent edges.
    \item If $e$ is assigned a colour $c$ and is not uncoloured in the previous
        step, uncolour it with probability $\Eq_i(e,c)$~(to be defined below).
    \item For each vertex $v$ and colour $c \in L_{i}(v)$, if $c$ is retained
        by {an edge} in $T_{i}(v,c)$, then remove $c$ from $L_i(v)$.
    \item For each vertex $v$ and colour $c \in L_{i}(v)$, if $c$ is not retained by an edge in $T_{i}(v,c)$, then remove $c$ from $L_i(v)$ with
        probability $\Vq_i(v, c)$~(to be defined below).
\end{enumerate}
As we pointed out before, the steps S.4 and S.6 correspond to the equalizing coin flips and are performed to correct the fluctuations of the process.

For an edge $e=uv$ and a colour $c \in L_{i}(e)$, let $\kp_{i}(e,c)$ denote the
probability that no edge adjacent to $e$ is assigned colour $c$. It follows
from Lemma~\ref{lem:i:star} and
Lemma~\ref{lem:L:T:close} that
\begin{align*}
\kp_i(e,c)
    &= \left(1-\frac{1}{L_i}\right)^{t_i(u,c)+ t_i(v,c) - 2} \\
    & > \exp\left(-\frac{t_i(u,c)+ t_i(v,c) - 2}{L_i}\right) -
    O\left(\frac{1}{L_i}\right) \\
    & \ge \exp\left(-\frac{2 T_i - 2}{L_i}\right) -
    O\left(\frac{1}{L_i}\right) \\
    & \ge \expo^{-2}\left(\exp\left\{ \frac {2\eps} {9+\varepsilon} - o(1) +
    \frac 2 {L_i}\right\} \right) -
    O\left(\frac{1}{L_i}\right) \\
    &> \expo^{-2}.
\end{align*}
The last inequality holds since $\Delta$ is large with respect to $\varepsilon$ and
Lemma~\ref{lem:i:star} shows that $L_i\geq \eps_1 \Delta$.

Hence, one can choose $\Eq_i(e,c):= 1 -
1/(\expo^{2} \kp_{i}(e,c)) > 0$.  This ensures that the probability that $e$ retains
$c$, {conditional on $c$ being assigned to $e$}, is precisely $(1-\Eq_i(e,c))\kp_{i}(e,c) = \expo^{-2}$.  Therefore, the expected value of $r_{i+1}(v,c)$, is precisely
{$(1-\expo^{-2})r_i(v,c)$}.

Similarly, for every colour $c \in L_i(v)$, the probability that the colour $c$
is not retained by an edge incident to $v$ is $\rem_{i}(v,c) = 1 - t_{i}(v,c) /
(\expo^2 L_i) > 1-\expo^{-2}$.
Thus choosing $\Vq_i(v,c) := 1 -
(1-\expo^{-2})/\rem_{i}(v,c)$ ensures that the probability that $c$ remains in
$L_i(v)$ is precisely $(1-\Vq_i(v,c))\rem_{i}(v,c) = 1 - \expo^{-2}$.

{Consider an edge $e = uv$. For a colour $c \in L_i(e)$ to be in
$L_{i+1}(e)$, it must be in both $L_{i+1}(u)$ and $L_{i+1}(v)$. If the two events $\{c \in
L_{i+1}(u)\}$ and $\{c \in L_{i+1}(v)\}$ were independent, then the expected value
of $\ell_{i+1}(v,c)$ would be precisely $(1-\expo^{-2})^2\,\ell_i(e)$. It can be
shown that this estimation is not far from being accurate.}

{For an edge $e \in T_i(v,c)$, in order that $e$ also belongs to $T_{i+1}(v,c)$, $e$
must stay uncoloured and $c$ must also be in $L_{i+1}(e)$. If the two events
\{$e$ is uncoloured\} and $\{c \in L_{i+1}(e)\}$ were independent then the expected
value of $t_{i+1}(v,c)$ would be precisely $(1-\expo^{-2})^2\,t_i(v,c)$. Although the assumption of independence is
not true, it turns out that the effect of dependency is small.}

The following lemma turns the above intuitions into precise statements. The
proofs are omitted as they are exactly the same as those in
\cite[chap~14.3]{Molloy2002graph}.

\begin{lem}\label{lem:expectation}
Suppose that $P(i-1)$ holds. Then, for every vertex $v$, every edge $e$ and every colour $c \in L_i(v)$,
    \begin{enumerate}[(a)]
    \item $\E{r_{i+1}(v,c)}  =   \left(1-\expo^{-2}  \right)r_i(v,c),$
    \item $\E{\ell_{i+1}(e)}  \geq   (1-\expo^{-2})^2\,L_i,$ and
    \item $\E{t_{i+1}(v,c)}  \leq   (1-\expo^{-2} )^2\,T_i + 1.$
    \end{enumerate}
\end{lem}

\subsection{The Concentration}\label{sec:con}

This subsection sketches the proof that $\ell_{i+1}(e)$, $t_{i+1}(v,c)$ and
$r_{i+1}(v,c)$ have highly concentrated distributions given that $P(i-1)$
holds.

\begin{lem}\label{lem:concentration}
    There exists a constant $\beta>0$ such that for every constant $\gamma>0$ and for every $1\leq i \le i^{*}$: \\
    (a) For every vertex $v$ and every colour $c \in L_i(v)$,
    $$
    \p{|r_{i+1}(v,c) - \E{r_{i+1}(v,c)}| \ge
    \gamma\sqrt{\E{r_{i+1}(v,c)}}\Delta^{1/6}}  \le \exp\{-\beta \gamma^2\Delta^{1/3}\};$$
    (b) For every edge $e$,
    $$
    \p{|\ell_{i+1}(e) - \E{\ell_{i+1}(e)}| >
     \gamma\sqrt{L_{i+1}}\Delta^{{1/6}}} \leq \exp\{-\beta
         \gamma^2\Delta^{1/3}\};
    $$
    (c) For every vertex $v$ and colour $c$,
    $$
        \p{|t_{i+1}(v,c) - \E{t_{i+1}(v,c)}| >
         \gamma\sqrt{L_{i+1}}\Delta^{{1/6}}} \leq \exp\{-\beta
             \gamma^2\Delta^{1/3}\}.
    $$
\end{lem}
The tool used to prove the above lemma is Talagrand's inequality
\cite[chap.~10]{Molloy2002graph}. It implies that to show a random variable $X$ is
concentrated, it suffices to verify that there exist two constants $\xi_1$
and $\xi_2$ such that:
\begin{enumerate}[(a)]
    \item changing the outcome of one random choice can affect $X$ by at
        most $\xi_1$; 
    \item for every $s$, if $X \ge s$, then there exists a set of at most
        $\xi_2 s$ random choices whose {outcomes} certify that $X \ge s$.
\end{enumerate}
If these two conditions hold, then $$\p{|X-\E{X}| > \alpha \sqrt{\E{X}}} \le
\exp\{-\beta \alpha^2\},$$ for $\beta < \frac 1 {10 \xi_1^2 \xi_2}$ and
$\alpha \le \sqrt{\E{X}}$. In the proof of the Lemma~\ref{lem:concentration},
the random choices are the colour assignments and equalizing coin flips. We can
take $\beta$ as the minimum of the $\beta$'s for the three types of
random variables and $\alpha=\gamma\Delta^{1/6}$.

For the part (a) of Lemma~\ref{lem:concentration}, note that changing the colour
assigned to an edge from $c_1$ to $c_2$ can only uncolour at most one edge
incident to $v$, because two or more incident edges assigned colour $c_2$ would
be uncoloured anyway. Meanwhile changing the outcome of one equalizing coin
flip can only change $r_{i+1}(v,c)$ by at most one. And for $s$ edges in
$R_{i}(v,c)$ which get uncoloured, either by conflicts or by equalizing coin
flips, at most $s$ outcomes of random choices are enough to certify it.

The proofs of parts (b) and (c) are more complicated. The idea is to split the random
variables into a linear combination of related random variables, which can be
shown to have concentrated distribution by Talagrand's inequality. The detailed
proofs are given in \cite[chap.~14.4]{Molloy2002graph}.

\subsection{Control at each iteration}\label{sec:ite}

In this section we will show that the property $P(i)$ defined in Section \ref{sec:outline} holds for
every $0 \leq i\leq i^*$ with positive probability.
We define the following set of events,
    \begin{align*}
        A^{\{c,d\}}_C &:=[\lambda_{i+1}(C) < \Lambda_{i+1}]
        & & \quad \text{for every pair of colours $c,d$ and
        every}\\
        & \;\;\;\;\cap [\text{$C$ is $\{c,d\}$-significant at the \((i+1)\)-st iteration}]
        & & \qquad  \text{cycle $C$ which is $\{c,d\}$-significant at}\\
        & & & \qquad  \text{the $i$-th iteration}, \\
        B_e &:=  [\ell_{i+1}(e) < L_{i+1}]  & & \quad \text{for  every edge $e$ uncoloured}, \\
        C_{v,c}&:= [t_{i+1}(v,c) > T_{i+1}] & & \quad \text{for  every vertex \(v\) and every colour $c$, and}\\
        D_{v,c} &:= [r_{i+1}(v,c) {>} R_{i+1}] & & \quad \text{for  every vertex \(v\) and every colour $c$}\;.
    \end{align*}
    Let $\cE$ be the collection of the above events. Observe that $\p{P(i)} =
\p{\cap_{E \in \cE} \overline{E}}$ so it suffices to lower bound the latter probability.

The next lemma bounds the probability of the events $B_e$, $C_{v,c}$ and $D_{v,c}$ {using the results obtained in Sections~\ref{sec:exp} and~\ref{sec:con}}.
\begin{lem}\label{lem:bounds}
    Given that {$P(i-1)$} holds, for every uncoloured edge $e$, every vertex $v$, and colour $c$
    \begin{enumerate}[(a)]
        \item $\p{B_e} = \expo^{-\Omega(\Delta^{1/3})}$;
        \item $\p{C_{v,c}} = \expo^{-\Omega(\Delta^{1/3})}$;
        \item $\p{D_{v,c}} = \expo^{-\Omega(\Delta^{1/3})}$.
    \end{enumerate}
\end{lem}
\begin{proof}
    For part $(a)$,
    recall that $L_{i+1} = (1-\expo^{-2})^2\, L_{i} - L_{i}^{2/3}$ by definition.
    By Lemma~\ref{lem:i:star}, $L_i\geq \eps_1\Delta$ and thus, there exists a constant $\gamma$ such that $L_i^{2/3}\geq \gamma \sqrt{L_{i+1}}
     \Delta^{\frac{1}{6}}$.
    Combining these results with the computation of $\E{\ell_{i+1}(e)}$ in Lemma~\ref{lem:expectation}
        \begin{align*}
            \p{B_e} & = \p{\ell_{i+1}(e) < L_{i+1}} \\
            &= \p{\ell_{i+1}(e) < (1-\expo^{-2})^2\,L_{i} - L_{i}^\frac{2}{3}} \\
        & \leq \p{\ell_{i+1}(e) < \E{\ell_{i+1}(e)} - \gamma \sqrt{L_{i+1}}\Delta^{\frac{1}{6}}}  \\
        & \le \exp\{-\beta \gamma^2 \Delta^\frac{1}{3}\},
        \end{align*}
        where the last inequality follows from Lemma~\ref{lem:concentration}.

    The proofs of the parts $(b)$ and $(c)$ are along the same lines of the previous one and use the bounds on the expected values derived in Lemma~\ref{lem:expectation} and the concentration inequalities provided in Lemma~\ref{lem:concentration}.
\end{proof}

Now we compute a bound on the probability of $A^{\{c,d\}}_{C}$.
\begin{lem}\label{lem:bounds2}
    {For every $0\leq i\leq i^*$, every pair of colours $\{c,d\}$ and every cycle $C$ that is $\{c,d\}$-significant at the $i$-th iteration,} we have
\begin{align*}
    \p{A^{\{c,d\}}_{C}}&{\leq} 
    \Lambda_{i+1} \lambda_i(C)^{2\Psi_{1}+\Lambda_{i+1}}
    L_i^{-(\lambda_i(C)-\Lambda_{i}/2+4 \Psi_1)} \;.
\end{align*}
Moreover, we also have
\begin{align*}
-\log_\Delta{\p{A_C^{c,d}}} &\geq \lambda_i(C)/3\;.
\end{align*}
\end{lem}
\begin{proof}
Recall that a cycle \(C\) is \( \{c,d\}\)-significant at
the \(i\)-th iteration, if at the beginning of
the $i$-th iteration, \(C\) is \( \{c,d\}\)-compatible, has $\{c,d\}$-multiplicity at most
\(\Psi_i\) and contains at most \(\Psi_{i}\)  \(
\{c,d\}\)-reserved edges. 
    If a cycle \(C\) is \( \{c,d\}\)-significant at the \( (i+1)\)-st iteration, it must be \(
    \{c,d\}\)-significant at the \(i\)-th iteration. Thus by property \(P(i-1)\), \(C\) contains at least \(\Lambda_{i}\)
    free edges at the beginning of the \(i\)-th iteration; that is, \(\lambda_{i}(C) \ge \Lambda_{i}\).
    Moreover, \(C\) contains at most \(\Psi_{i}\) reserved edges.
    For \(C\) to stay significant, after the \(i\)-th
    iteration $C$ must have multiplicity at most \(\Psi_{i+1}\). 
 This implies that the cycle $C$ can be partitioned into at most \(\Psi_{i+1}\)
        paths that are either $(c,d)$-alternating or $(d,c)$-alternating. 
%
    Thus there are at most \(\lambda_i(C)^{\Psi_{i+1}}\) ways to partition the
        free edges into \(\Psi_{i+1}\) groups according to which alternating path they belong to.
        Once this partition is decided, for each such group one should choose if they form a  \( (c,d)\)-alternating path or a \( (d,c)\)-alternating.
        Using that $\lambda_i(C)\geq 2$ and that $\Psi_{i+1}\leq \Psi_1$, it follows that there are at most \((2\lambda_{i}(C))^{\Psi_{i+1}} \le \lambda_{i}(C)^{2 \Psi_1}\) ways to decide
        which edge of $C$ should have colour \(c\) or \(d\) such that \(A_{C}^{\{c,d\}}\) holds.

    Another condition for \(C\) to keep being significant is that no reserved edge in \(C\) gets
    coloured. As only an upper bound of \(\p{A_{C}^{\{c,d\}}}\) is needed, we can assume this
    condition is satisfied.

Thus we can upper bound the probability of $A^{\{c,d\}}_{C}$ using a union bound over
all the possible sets of free edges of size at most $\Lambda_{i+1}$ and over the at most
    \( \lambda_{i}(C)^{2\Psi_{1}}\) possible ways to
    partition the \(\lambda_{i}(C)\) free edges into alternating paths.
Recall that at the beginning of the $i$-th iteration the cycle $C$ has $\lambda_i(C)$ free
edges, and that the list of each uncoloured edge has size exactly $L_i$. Also recall that
$\Lambda_{i}:=2k/2^{i-1}-4 \ctn i$ and that \(\Psi_i:= 4^{2+i^{*}-i}\).
We have
\begin{align*}
    \p{A^{\{c,d\}}_{C}}
    & {\leq}
    \lambda_{i}(C)^{2\Psi_1} \sum_{t=0}^{\Lambda_{i+1}}
    \binom{\lambda_i(C)}{t} L_i^{-(\lambda_i(C)-t)}\\
    &
    {\leq}
   \Lambda_{i+1}
    \lambda_i(C)^{2\Psi_1+\Lambda_{i+1}} L_i^{-(\lambda_i(C)-\Lambda_{i+1})} \\
    &
    {=}
    \Lambda_{i+1}
    \lambda_i(C)^{2\Psi_1+\Lambda_{i+1}}
    L_i^{-(\lambda_i(C)-\Lambda_{i}/2)-(\Lambda_{i}/2-\Lambda_{i+1})}
    \\
    &
    {\leq}
    \Lambda_{i+1}
    \lambda_i(C)^{2\Psi_1+\Lambda_{i+1}}
    L_i^{-(\lambda_i(C)-\Lambda_{i}/2)-\Psi_{1}\left( 2i +4 \right)}
    \\
    &
    {\leq}
    \Lambda_{i+1}
    \lambda_i(C)^{2\Psi_1+\Lambda_{i+1}}
    L_i^{-(\lambda_i(C)-\Lambda_{i}/2+4 \Psi_1)}
    \;.
\end{align*}
This proves the first part of the lemma. For the second part, using $\sum_{i=0}^b \binom{a}{i}\leq 2^a$, we obtain,
\begin{align*}
    \p{A^{\{c,d\}}_{C}}
    &
    {\leq}
    \lambda_i(C)^{2\Psi_1} 2^{\lambda_i(C)}
    L_i^{-(\lambda_i(C)-\Lambda_{i}/2)} \\
    &
    {\leq}
    \lambda_i(C)^{2\Psi_1} 2^{\lambda_i(C)}
    L_i^{-(\lambda_i(C)/2)}\;,
\end{align*}
where we used that by induction hypothesis $P(i-1)$ holds, and thus $\lambda_i(C)\geq\Lambda_i$.

Recall that $L_i\geq \eps_1 \Delta$ by Lemma~\ref{lem:i:star}. Hence,
\begin{align*}
-\log_\Delta{\p{A_C^{c,d}}} &\geq \frac{\lambda_i(C)}{2}\cdot \left(\log_\Delta{L_i}-2\log_\Delta{2}\right) - 2\Psi_1\log_\Delta{\lambda_i(C)} \\
&\geq \frac{\lambda_i(C)}{2}\cdot \left(1 + \log_\Delta{(\eps_1/4})\right) - 2\Psi_1\log_\Delta{\lambda_i(C)} \\
&\geq \frac{\lambda_i(C)}{3}\;,
\end{align*}
provided that $\d$ and $g$ (and thus, $k$) are large enough with respect to $\eps$.
\end{proof}

We will use the following modified version of the Weighted Lov\'asz Local Lemma~(see, e.g.,~\cite[p.~221]{Molloy2002graph})
to lower bound $\p{\cap_{E \in \cE} \overline{E}}$. Its proof can be {easily} derived from the general version of the
local lemma~(see, e.g.,~\cite[p.~222]{Molloy2002graph}), and we include it here for the sake of
completeness:
{
\begin{lem}[General Lov\'asz Local Lemma]
    Consider a set $\cE = \{E_1 ,
    \ldots ,E_M\}$ of events such that each $E_r$ is mutually independent of
    $\cE\setminus (\cD_r \cup \{E_r\})$, for some $\cD_r \subseteq \cE$. If there
    exist constants $x_1,\ldots,{x_M} \in [0,1)$ such that for each $1 \le r \le M$
$$
\p{E_r}\leq x_r  \prod_{E_s \in \cD_r} (1-x_s)\;,
$$
    then with positive probability, none of the events in $\cE$ occur.
\end{lem}
}
\begin{lem}[Alternative version of the Weighted Lov\'asz Local Lemma] \label{lem:local_lemma}
Let $\beta > 0$ be a
    constant and let {$\nu>0$ be} such that $\beta = \expo^{2\log{2}/\nu}$.
    Consider a set $\cE = \{E_1 ,
    \ldots ,E_M\}$ of events such that each $E_r$ is mutually independent of
    $\cE\setminus   (\cD_r \cup \{E_r\})$, for some $\cD_r \subseteq \cE$. If there
    exist {constants} $w_1,\ldots,{w_M} \ge 1$ and {$h_1,\ldots, h_M$ such that}
    \vspace{-2mm}
    \begin{enumerate}[(a)]
        \item $\p{E_r} \leq h_r$;
        \item $\sum_{E_s \in \cD_r}  \beta^{w_s} h_s \le  {w_r}/{\nu}$;
        \item $\beta^{w_r} h_r \le  1/2$;
    \end{enumerate}

    then with positive probability, none of the events in $\cE$ occur.
\end{lem}

\begin{proof}
    Let $x_r := \beta^{w_r} h_r$. By assumption $x_r \le 1/2$. Thus $1-x_r \ge
\expo^{-\alpha x_r}$, where $\alpha = 2\log{2}$. Observe also that $\beta=\expo^{\alpha/\nu}$.
    Therefore
    \begin{align*}
        x_r \prod_{E_s \in \cD_r} (1-x_s)
        & \ge x_r \exp\left\{-\alpha \sum_{E_s \in \cD_r} x_s  \right\} \\
        & = \beta^{w_r} h_r \exp \left\{ -\alpha \sum_{E_s \in \cD_r}
        \beta^{w_s}h_s \right\} \\
        & \ge \beta^{w_r} h_r \exp \left\{ - \frac {\alpha w_r} \nu
    \right\} \\
        & = h_r \ge \p{E_r}.
    \end{align*}
    Thus the lemma follows from the General Lov\'asz Local Lemma.
\end{proof}

%

For every $\lambda\geq 0$, every edge $e$ and every pair of colours $\{c,d\}$, let \(\cC^{\{c,d\}}(e,\lambda)\) be the set of \( \{c,d\}\)-significant cycles at the \(i\)-th
iteration that contain the edge $e$, that contain exactly \(\lambda\) free edges, and
that can possibly remain significant {at the \((i+1)\)-st iteration}.
Let $\cC^{\{c,d\}}(e)=\cup_{\lambda\geq 0} \cC^{\{c,d\}}(e,\lambda)$.
The following lemma controls the contribution of the probability of the events $A^{\{c,d\}}_C$ for all
\(C \in \cC^{\{c,d\}}(e)\).

\begin{lem}
\label{lem:cycle:incremental}
    Suppose that $P(i-1)$ holds. Then, for every pair of colours $\{c,d\}$ and every free \(c\)-compatible
    edge $e$,
\begin{align}\label{eq:for_the_LLL}
\sum_{C\in \cC^{\{c,d\}}(e)} \left(1+\frac{\eps}{10}\right)^{\lambda_i(C)/2} \p{A^{\{c,d\}}_{C}} &=
O(\d^{-7})\;.
\end{align}
\end{lem}
\begin{proof}



    We count the number of cycles that could possibly remain \( \{c,d\}\)-significant after the \(i\)-th iteration and that contain a given {free} edge
$e=uv$, by constructing them from the edge $e$. At each step of the
construction, we consider how many of the unused edges can extend the current cycle. We will use the
fact that, since $P(i-1)$ holds, every such cycle has many free edges.

Consider the graph $G'$ obtained from $G$ by keeping all the edges that participate in cycles
    that could be $\{c,d\}$-significant at the $(i+1)$-st iteration. In particular, all these edges are $\{c,d\}$-compatible.
We are going to construct a subgraph $H\subseteq G'$ rooted at $v$. Let $V_0:=\{v\}$ and
let $H_0:=G'[V_0]$ be the empty subgraph formed by $v$. We construct $H_{j+1}:=G'[V_{j+1}]$, with
$V_{j+1}:=V_j\cup W_{j+1}$, where $W_{j+1}$ is the set of vertices {$w\in V(G)\setminus V_j$} that
satisfy
\begin{enumerate}[(a)]
    \item $w$ is at distance \(1\) from \(W_{j}\) in $G' \setminus e$, and

    \item there exists a  path in $G' \setminus e$
        connecting \(v\) to \(w\) 
        using exactly one vertex in each of \(W_{0},\ldots, W_{j}\)
        that contains:
        \begin{enumerate}[({b}.1)]
            \item at most $\Lambda_{i+1}$ free edges,
            \item has $\{c,d\}$-multiplicity at most \(\Psi_{i+1}\), and
            \item at most \(\Psi_{i+1}\) \( \{c,d\}\)-reserved edges.
        \end{enumerate}
\end{enumerate}

\noindent Let $\tau$ be the smallest integer such that {$W_{\tau+1}=\emptyset$} and consider $H:=G'[V_\tau]$.

\begin{claim}
For every $0\leq j < \tau$, each vertex in $W_{j+1}$ has a unique neighbour in $H_{j+1}$ which belongs to $W_j$. In particular, the subgraph $H$ is a rooted tree.
\end{claim}
\begin{proof}
Equivalently, we want to show that for every vertex $w\in H$ there is a unique path in \(H\) to the root $v$.
For the sake of contradiction, suppose that there exists a cycle {in $H$}. Let $j$ be the smallest integer
such that $H_j$ contains at least one cycle. 

Adding \(W_{j}\) can only create a cycle in one of two ways: either (1) there exists \(w \in W_{j}\) such that \(w\) is
connected two vertices in \(H_{j-1}\), or (2) there exist \(w_1, w_2 \in W_{j}\) such that there is a path from \(w_1\) to \(w_2\) in \(G'[W_j]\). If (1) occurs, we can find a cycle \(D\) in \(H_{j}\) that consists of two internally-disjoint 
paths \(p_1\) and \(p_2\) contained
in \(H_{j-1}\), and the vertex $w$.
If (1) does not occur, for every $w\in W_j$, there exists a unique path in $H_{j-1}$ that connects
the unique neighbour of $w$ in $W_{j-1}$ to $v$. Therefore, we can assume that $w_1w_2$ is an edge
of $G'$ (otherwise take any of the edges in the path between $w_1$ and $w_2$ in $G'[W_j]$) and we
obtain a cycle $D$ formed by  two internally-disjoint paths $p_1$ and $p_2$ contained in
\(H_{j-1}\) and the vertices $w_1$ and $w_2$.


%

By our construction, \(p_1\) and \(p_2\) have at most one vertex in each set $W_i$, for $i<j$, and
thus, each of them has fewer than \(\Lambda_{i+1}\) free edges, multiplicity at most
\(\Psi_{i+1}\) and at most \(\Psi_{i+1}\) reserved edges.  Since $D$ is composed of $p_1$, $p_2$ and
at most three extra edges (the one connecting $p_i$ to $w_i$, $i\in \{1,2\}$ and $w_1 w_2$), it must
have at most \(2\Lambda_{i+1} + 3 \le \Lambda_{i}\) free edges, multiplicity at most \(2\Psi_{i+1} +
    3 \le \Psi_{i}\) and at most \(2\Psi_{i+1} + 3 \le \Psi_{i}\) reserved edges.  Thus $D$ was
$\{c,d\}$-significant at the \(i\)-th iteration but contains fewer than $\Lambda_{i}$ free edges,
thus giving a contradiction with the fact that every \( \{c,d\}\)-significant cycle has at least
$\Lambda_i$ free edges, since $P(i-1)$ is satisfied. Hence, there exists a unique path in $H$ from any vertex $w$ to the root $v$. This implies that every vertex in $W_{j+1}$ has a unique neighbour in $H_{j+1}$ (which belongs to $W_j$) and that $H$ is acyclic.
\end{proof}

Recall that for a cycle $C$ to be able to be \( \{c,d\}\)-significant at the \( (i+1)\)-st iteration,
that is \(C\in\cC^{\{c,d\}}(e)\), it must satisfy the following three properties:
\begin{itemize}
\item[-] $C$ is \(\{c,d\}\)-significant at the \(i\)-th iteration,
\item[-] $C$ has $\{c,d\}$-multiplicity at most \(\Psi_{i+1}\): during an iteration the multiplicity of a cycle can only increase, and
\item[-] $C$ contains at most \(\Psi_{i+1}\) $\{c,d\}$-reserved edges: the number of reserved edges only decreases when one of them retains a colour
different than $c$ or $d$, and in such case, the edge is neither $c$- nor $d$-compatible, and the $C$ is no longer \( \{c,d\}\)-significant.
\end{itemize}
%

By construction of $G'$, if \(C \in \cC^{\{c,d\}}(e)\), then $C$ is contained in \(G'\).  Since $C$
contains the edge $e$ and since $H$ is acyclic, there exists a simple path $P\subset C$ in $H$ from
\(v\) to a leaf \(w'\) of $H$ such that $P$ intersects each set $W_i$ exactly once. (Observe that
$C$ might intersect $H$ in a very complicated way, but the path $P$ has a simple structure.) As we
explained before, if \(C\in \cC^{\{c,d\}}(e)\), then it has $\{c,d\}$-multiplicity at most
\(\Psi_{i+1}\) and contains at most \(\Psi_{i+1}\) $\{c,d\}$-reserved edges. Fix $C$ and create $P$
at the same time we construct $H$. Since conditions (b.2) and (b.3) are always satisfied by the
properties of $C$, during the construction of \(H\) the path $P$ can only stop growing because
condition (b.1) is violated, which implies that $P$ contains exactly \(\Lambda_{i+1}\) free edges.
(The endpoint \(w'\) of \(P\) must be a leaf in \(H\) since no neighbour of \(w'\) can be
    added to \(H\) any more.) Thus, $C\setminus P$ must be a path that contains exactly
\(\lambda_{i}(C)-\Lambda_{i+1}\) free edges,  has multiplicity at most \(\Psi_{i+1}\) and contains
at most \(\Psi_{i+1}\) reserved edges.

Let \(P_u(\lambda_f,\lambda_r,\lambda_s)\) be the set of paths starting from \(u\) and containing
\(\lambda_f\) free edges, \(\lambda_r\) reserved edges and multiplicity \(\lambda_s\). There are
\(\binom{\lambda_f+\lambda_r}{\lambda_r} \le (\lambda_f+\lambda_r)^{\lambda_r}\) ways to select
which edges are reserved and which ones are free in the path. Once this is fixed, there are at most
\( (2(\lambda_f + \lambda_r))^{\lambda_s}\) ways to partition free and reserved edges into groups according to which alternating path they belong to.  Given the previous selections, there are at most
\(T_{i}^{\lambda_f}\Delta^{\lambda_r}\) paths starting from \(u\) that agree with them: a vertex
can be incident either to one coloured edge, to at most \(T_{i}\) free edges or to at most
\(\Delta\) reserved edges~(the colour corresponding to these edges is either $c$ or $d$
depending on their distance to \(e\)).  Therefore, summing over all the possible selections, we have
\(|P_u(\lambda_f,\lambda_r,\lambda_s)| \le 2^{\lambda_s}(\lambda_f+\lambda_r)^{\lambda_r+\lambda_s}
T_{i}^{\lambda_f} \Delta^{\lambda_r}\).

%
Thus, the number of candidate paths $P$ that are induced by some cycle \(C \in \cC^{\{c,d\}}(e, \lambda)\) is at most
\begin{align*}
    \sum_{\lambda_r = 0}^{\Psi_{i+1}}
    \sum_{\lambda_s = 0}^{\Psi_{i+1}}
    \left|P_u\left(\lambda-\Lambda_{i+1},\lambda_r,\lambda_s\right)\right|
    &
    \le
    \sum_{\lambda_r = 0}^{\Psi_{i+1}}
    \sum_{\lambda_s = 0}^{\Psi_{i+1}}
    2^{\lambda_s}
    \left(\lambda-\Lambda_{i+1}+\lambda_r\right)^{\lambda_r+\lambda_s}
    T_{i}^{\lambda-\Lambda_{i+1}} \Delta^{\lambda_r}
    \\
    &
    \leq
    2^{\Psi_{i+1}}
    \lambda^{4\Psi_{i+1}}
    T_{i}^{\lambda-\Lambda_{i+1}}
    \Delta^{\Psi_{i+1}}
    \leq
   2^{\Psi_{i+1}}
    \lambda^{4\Psi_{1}}
    T_{i}^{\lambda-\Lambda_{i+1}}
    \Delta^{\Psi_1}
    ,
\end{align*}
where we used that $\Lambda_{i+1}\geq \Psi_{i+1}\geq \lambda_r$, since $k$ is large enough with respect to $\eps$ (and thus, to $i^*$).

This is also an upper bound on \(|\cC^{\{c,d\}}(e,\lambda)|\): since $C\setminus P$ must start at $u$ and
end at a leaf of \(H\), once this leaf is determined and	 since $H$ is a tree, \(C\) is also determined.

Recall that if a cycle $C$ is not $\{c,d\}$-significant, then $\p{A_C^{\{c,d\}}}=0$, and it does not contribute to the final sum.
By Lemma~\ref{lem:i:star}, $L_i\geq \eps_1\Delta$ and by Lemma~\ref{lem:L:T:close},
$L_i/T_i\geq 1+\eps/10$.
Using the upper bound obtained in Lemma~\ref{lem:bounds2} we have,
\begin{align*}
    \sum_{C\in \cC^{\{c,d\}}(e)} \left(1+\frac{\varepsilon}{10}\right)^{\lambda_i(C)/2} \p{A^{\{c,d\}}_{C}}
    &=
    \sum_{\lambda\geq \Lambda_i}
    \sum_{C\in \cC^{\{c,d\}}(e,\lambda)}\left(1+\frac{\eps}{10}\right)^{\lambda/2} \p{A^{\{c,d\}}_{C}}\\
    &{\leq}
    \sum_{\lambda\geq \Lambda_i}\left(1+\frac{\eps}{10}\right)^{\lambda/2} |\cC^{\{c,d\}}(e,\lambda)| \Lambda_{i+1}\lambda^{2\Psi_{1}+\Lambda_{i+1}}L_i^{-(\lambda-\Lambda_{i+1}+4 \Psi_1)}
    \\
    &{\leq}
    \sum_{\lambda\geq \Lambda_i}
    2^{\Psi_{i+1}}
    \left(1+\frac{\eps}{10}\right)^{\lambda/2}  
    \lambda^{6\Psi_{1}+\Lambda_{i+1}+1}
    \Delta^{\Psi_{1}}
    \frac
    {
      T_{i}^{\lambda-\Lambda_{i+1}}
    }
    {
        L_i^{4\Psi_1} L_{i}^{\lambda-\Lambda_{i+1}}
    }
    \\
    &{\leq}
    \sum_{\lambda\geq \Lambda_i}
    2^{\Psi_{i+1}}
    \left(1+\frac{\eps}{10}\right)^{\lambda/2} \lambda^{6\Psi_{1}+\Lambda_{i+1}+1}
    \frac{\Delta^{\Psi_1}}{L_i^{4\Psi_1}}
    \left(
    \frac
    {
        T_i
    }
    {
        L_{i}
    }
    \right)^{\lambda-\Lambda_{i+1}}
    \\
    &
    {\leq}
    \Delta^{-3\Psi_{1}}
    \sum_{\lambda\geq \Lambda_i}
    2^{\Psi_{i+1}}
    \lambda^{6\Psi_{1}+\Lambda_{i+1}+1}
    \eps_1^{-4\Psi_1}\left(1+\frac{\eps}{10}\right)^{-(\lambda/2-\Lambda_{i+1})} \\
    &
    {\leq}
    \Delta^{-3\Psi_{1}}
    \sum_{\lambda\geq \Lambda_i}
    2^{\Psi_{i+1}}
    \lambda^{6\Psi_{1}+\Lambda_{1}+1}
    \eps_1^{-4\Psi_1}\left(1+\frac{\eps}{10}\right)^{-(\lambda/2-\Lambda_1)}, 
\end{align*}
where we have used that $L_i\geq \eps_1\Delta$ (Lemma~\ref{lem:L:T:close}).

Let $\lambda_0$ be the smallest integer $\lambda$ such that
$$
     2^{\Psi_{i+1}}\lambda^{6\Psi_{1}+\Lambda_{1}+1} \eps_1^{-4\Psi_1}\left(1+\frac{\eps}{10}\right)^{-(\lambda/4-\Lambda_1)}\leq 1\;.
$$
This value clearly exists since the LHS of the previous inequality goes to 0 when $\lambda\to +\infty$. Observe that in such a case,
$$
\sum_{\lambda\geq \lambda_0}      2^{\Psi_{i+1}}
 \lambda^{6\Psi_{1}+\Lambda_{1}+1}
    \eps_1^{-4\Psi_1}\left(1+\frac{\eps}{10}\right)^{-(\lambda/2-\Lambda_1)}\leq \sum_{\lambda\geq \lambda_0} \left(1+\frac{\eps}{10}\right)^{-\lambda/4}= O(1)\;.
$$
Notice that the $\eps_1$ given in Lemma~\ref{lem:L:T:close} only depends  on $\eps$. Let $\Delta$ be large enough such that
$$
\Delta \geq \Delta_0(\eps,k)= \sum_{\lambda=0}^{\lambda_0-1}       2^{\Psi_{i+1}}
\lambda^{6\Psi_{1}+\Lambda_{1}+1}
    \eps_1^{-4\Psi_1}\left(1+\frac{\eps}{10}\right)^{-(\lambda/2-\Lambda_1)}\;.
$$
Then,
\begin{align*}
 \sum_{C\in \cC^{\{c,d\}}(e)}  \p{A^{\{c,d\}}_{C}}&\leq 
 \d^{-3\Psi_1}  (\Delta_0(\eps,k)+ O(1))=  O(
 \d^{-7}) \;,
\end{align*}
since $\Psi_1\geq 3$.
\end{proof}

\begin{lem}\label{lem:applying_local_lemma}
    With positive probability, $P(i)$ holds for every $0\leq i\leq i^*$.
\end{lem}
\begin{proof}
    This can be proved by induction. The basis of the induction is that
    $P(0)$ holds, which does deterministically by Lemma~\ref{lem:P(0)}.
    Now consider $i \geq 1$. By
    induction hypothesis, $\p{P(i-1)}> 0$.  Thus it suffices to show that
    $\p{P(i)} > 0$, provided that $P(i-1)$ holds.
    We will apply Lemma~\ref{lem:local_lemma} to $\cE$ in order to prove that
    $$
    \p{P(i)} = \p{\cap_{E \in \cE} \overline{E}} > 0.
    $$

     Let $\beta:=(1+\eps/10)^{1/2}$. Then $\nu=
     \frac{2\log{2}}{\log\beta} =\frac{4\log{2}}{\log(1+\eps/10)} $.
    For a type $A$ event $E$, we define its weight
    $w:=-{\log_\Delta(\p{E})}> 0$ and we set  $h:=\Delta^{-w}=\p{E}$.  Observe
    also that, 
    \[
        \beta^{w} h
        = \beta^{-\log_{\Delta}\p{E}}\cdot \p{E}
        = 
        (\p{E})^{1-\log_{\Delta}\beta}
        \le \sqrt{\p{E}}
        \le 1/2\;,
    \]
when $\Delta$ and \(g\) are large enough with respect to \(\beta\).

    For an event $E$ of other type, define its weight $w:=1$ and set
    $h:= \expo^{-\Omega(\Delta^{1/3})}$. Then ${\p{E}} \le h$ by
    Lemma~\ref{lem:bounds}. Also note that $\beta^{w} h = \beta
    \expo^{-\Omega\left( \Delta^{1/3} \right)} \le 1/2$, when $\Delta$ is large
    enough.

    Thus, it remains to show that for every event $E_r$ the sum 
    $\sum_{E_s\in \cD_r}\beta^{w_s}h_s$ is not
    too large.  Consider the case where $E_r=A_C^{\{c,d\}}$ for some
    {$\{c,d\}$-{significant}} cycle $C$. Let us first construct the set $\cD_r$. For any
    other cycle $C'$ and pair of colours $(c',d')$, $A_{C'}^{(c',d')}\in\cD_r$
    if and only if $C$ and $C'$ intersect in at least one {free} edge. For
    an event $E_s$ of type $B_e$, $C_{v,c}$ or $D_{v,c}$, $E_s\notin \cD_r$ if it
    corresponds to vertices or edges at distance at least $4$ from any
    {free} edge of $C$.


    {Recall that, by the second part of Lemma~\ref{lem:bounds2}, we have that if $E_r=A^{\{c,d\}}_C$
for some cycle $C$, then $\omega_r \geq \lambda_i(C)/3$.}
    We use Lemma~\ref{lem:cycle:incremental} on each {free} edge $e\in E(C)$ in order to bound the following sum
    \begin{align*}
    \sum_{E_s \in \cD_r} \beta^{w_s} h_s & =
    \sum_{(c',d')}\sum_{\substack{e\in E(C) \\ e \text{ {free}}}}\sum_{C'\in \cC^{(c',d')}(e)} \beta^{w_s} \p{A_{C'}^{(c',d')}}
    +\sum_{\substack{E_s \in \cD_r\\ E_s \text{ not type } A}} \beta^{w_s} h_s\\
    &\leq  \sum_{(c',d')}\sum_{\substack{e\in E(C) \\ e \text{ {free}}}} O(\Delta^{-7})+
    \sum_{\substack{E_s \in \cD_r\\ E_s \text{ not type } A}} \beta \expo^{-\Omega\left( \Delta^{1/3} \right)} \\
    &= \lambda_{i}(C) O(\Delta^{-5})+\lambda_{i}(C)O\left(\Delta^4 \expo^{-\Omega(\Delta^{1/3})}\right)\\
    &= \lambda_{i}(C)O(\Delta^{-5})
    \le  \frac{w_r}{\nu}\;.
    \end{align*}


    Now consider an event $E_r = B_e$. For some event $A_{C'}^{(c',d')}$ to depend on $E_r$, the
    cycle $C'$ must go through a {free} edge
    within distance $4$ to $e$. Therefore
    \begin{align*}
        \sum_{E_s \in \cD_r} \beta^{w_s} h_s & =
        \sum_{\substack{f \text{ within} \\ \text{distance $4$ of $e$}}}
        \sum_{(c',d')} \sum_{C'\in \cC^{(c',d')}(f)} \beta^{w_s} \p{A_{C'}^{(c',d')}}
        +\sum_{\substack{E_s \in \cD_r\\ E_s \text{ not type } A}} \beta^{w_s} h_s\\
        &\leq  \sum_{\substack{f \text{ within} \\ \text{distance $4$ of $e$}}} \sum_{(c',d')} O(\Delta^{-7})+
        \sum_{\substack{E_s \in \cD_r\\ E_s \text{ not type } A}} \beta \expo^{-\Omega\left( \Delta^{1/3} \right)} \\
        & = O(\Delta^{-1}) + O(\Delta^{4} \expo^{-\Omega(\Delta)})  \le
        \frac {w_r} \nu.
    \end{align*}

    The corresponding inequalities for events of type $C_{v,c}$ or
    $D_{v,c}$ can be shown in a similar way. Hence, it follows from
    Lemma~\ref{lem:local_lemma} that $P(i)$ is satisfied with positive
    probability. 
\end{proof}

\begin{proof}[Proof of Lemma~\ref{lem:pre_colouring}]
 From Lemma~\ref{lem:applying_local_lemma}, $P(i^*)$ holds with positive probability. Thus, the lemma follows directly from Lemma~\ref{lem:finalP}.
\end{proof}

\section{Proof of Lemma~\ref{lem:pre_finishing}}\label{sec:finish}
In this section we prove that the partial colouring obtained
at the end of the iterative colouring procedure can be extended to a complete acyclic edge colouring of \(G\).
\begin{proof}[Proof of Lemma \ref{lem:pre_finishing}]
    {The idea of the proof is to assign to each uncoloured edge a colour chosen uniformly at
    random from its list and to show that with positive probability, the obtained colouring satisfies the desired properties.}

    Let $L := \gamma \Delta$ and $T := {\gamma^2 \Delta/128}$.
    For every two adjacent uncoloured edges $e,f$ and every colour $c$, let $B_{e,f}^c$ denote the
    event that $e,f$ are both assigned colour $c$. {Consider a total order of the set of edges $E(G)$}. For every cycle $C$ and every pair
    of colours $\{c,d\}$, let $A^{\{c,d\}}_{C}$ denote the event that $C$ is bicoloured using the pair of colours $\{c,d\}$,
    {such that the smallest edge in $C$ is coloured with $c$}. Let $\cE$ denote the set of these events.

    Let $p=:1/L$. If an event $E_r$ is of type $B$, let $w_r := 2$. Then,
    $$\p{E_r} \le \left(\frac 1 L\right)^{2} = p^{w_r}.$$
    For every cycle $C$, we denoted by $\lambda(C)$ its number of uncoloured edges.
    If an event $E_r$ is of type $A$ and corresponds to a cycle a $C$, let $w_r := \lambda(C)$. As before,
    $$\p{E_r} \le \left(\frac 1 L\right)^{\lambda(C)} = p^{w_r}.$$

    By the hypothesis of the lemma, every cycle that can become completely bicoloured using the pair of colours $\{c,d\}$,
    has at least {$3$} uncoloured edges.
    Then, for every uncoloured edge $e$, every pair of colours $\{c,d\}$ and every $\lambda\geq 3$,
    there are at most $T^{\lambda-{2}}$ cycles $C$ that contain edge $e$ and that can become bicoloured. For the cycles that cannot become bicoloured, we have
    $\p{A^{\{c,d\}}_{C}} = 0$, and they do not have any contribution. As in Lemma~\ref{lem:cycle:incremental}, we observe that when {we construct} a cycle starting at the edge $e$, at every step we either select a coloured edge with the right colour (at most one option) or we {select one of the} at most $T$ different uncoloured edges that have the right colour on their reserved set of colours.  The $-2$ in the exponent is standard in these type of countings, and comes from the fact that $e$ is uncoloured and that the last uncoloured edge selected in the cycle has only one way to be chosen instead of $T$.

    {As there are in total at most \( (1+\varepsilon) \Delta\) colours and an {uncoloured} edge can choose only
        \(L\) of them, there are at most \( (1+\varepsilon) \Delta L < 2 \Delta L\) choices of colours $\{c,d\}$  such that
        \(\p{A^{\{c,d\}}_C}>0\) for some cycle \(C\) going through $e$.}

    Notice that if
    two cycles do not share an uncoloured edge, their corresponding type $A$
    events are independent.
    Also note that the colour an edge is assigned, depends on at most $2 L
    T$ events of type $B$.
    {Let $\cC^{\{c,d\}}(e,\lambda)$ be the set of cycles $C$ that contain edge $e$ that can be bicoloured using the pair of colours $\{c,d\}$.}
    Therefore, if
    $E_r=A_C^{\{c,d\}}$ and $\cD_r$ is a set of events such that $\cE\setminus
    \cD_r$ is mutually independent from $E_r$, then
    \begin{align*}
        \sum_{E_s \in \cD_r} (2p)^{w_s}
        &
        \le \sum_{\substack{e \in E(C): \\ \text{$e$ uncoloured}}} \left( \frac 2 L \right)^2 2 L T
        + \sum_{\substack{e \in E(C): \\ \text{$e$ uncoloured}}} \sum_{(c',d')} \sum_{\lambda \ge {3}}
        \sum_{\substack{C' \in \cC^{(c',d')}(e, \lambda):\\\p{A_{C'}^{(c',d')} \ge 0}}}
        \left( \frac 2 L \right)^\lambda \\
        &
        \le \lambda(C) \left( \frac 2 L \right)^2 2 L T
        + \lambda(C) \cdot {2 \Delta}L \sum_{\lambda \ge {3}}  T^{\lambda {-2}} \left( \frac
        2 L \right)^\lambda \\
        &
        = 8{\lambda(C)} \cdot \frac{T}{L}
        +
        8\lambda(C)
        \cdot\frac
            {\Delta}
            {L}
            \sum_{\lambda \ge {3}}
                \left(\frac {{2}T} {L}  \right)^{\lambda-2} \\
        &
        < \frac{\gamma\lambda(C)}{32} + \frac{8\lambda(C)}{64-\gamma}
        < \frac {\lambda(C)} 2
        = \frac {w_r} 2 \;.
    \end{align*}

    If $E_r=B^c_{e,f}$ for a pair of edges $(e,f)$ and a colour $c$, then it is independent to all but $4 L T$ events of type $B$. Therefore
    \begin{align*}
        \sum_{E_s \in \cD_r} (2p)^{w_s}
        &
        \le \left( \frac 2 L \right)^2 4 L T
        + \sum_{(c',d')} \sum_{\lambda \ge {3}}
        \sum_{\substack{C' \in \cC^{(c',d')}(e, \lambda) \cup \cC^{(c',d')}(f, \lambda):\\\p{A_{C'}^{(c',d')} \ge 0}}}
        \left( \frac 2 L \right)^\lambda \\
        &
        \le
        \left( \frac 2 L \right)^2 4 L T
        +
        {2 \Delta}L \sum_{\lambda \ge {3}}  2 T^{\lambda {-2}} \left( \frac
        2 L \right)^\lambda
        \\
        &
        =  16\cdot \frac{T}{L}
        +  16 \cdot
        \frac
            {\Delta}
            {L}
            \sum_{\lambda \ge {3}}
                \left(\frac {2T} {L}  \right)^{\lambda-2} \\
        &
        \le
        \frac{\gamma}{8}
        +
        \frac{16}{64-\gamma}
        < 1
        = \frac {w_r} 2 \;.
    \end{align*}
    We use the original version of the Weighted Lov\'asz Local Lemma \cite[p.~221]{Molloy2002graph} to conclude that, with positive probability, none of the events in $\cE$ happen.
\end{proof}
\vspace{-0.1cm}
Now, let {$\{S_v:\, v\in V(G)\}$} be the collection of {subsets of $[(1+\eps)\Delta]$ that
satisfies} properties (A.1)--(A.3) whose existence is proved in
    Lemma~\ref{lem:pre_reservation}.
Let \(\chi\) be a partial edge colouring of \(G\) with properties (B.1)--(B.3) that is guaranteed to
exist by Lemma \ref{lem:pre_colouring} using this collection of sets. Let \(\gamma := \frac{\varepsilon^2}{18}\).  {Assign to} each
uncoloured edge \(e = uv\) in \(\chi\) a list of \(\gamma \Delta\) colours chosen arbitrarily from
\(S_u \cap S_v\). This is possible since by property $(A.2)$ in Lemma~\ref{lem:pre_reservation},
these sets have size at least $\gamma \Delta$. Then, properties (B.1)--(B.3) of \(\chi\) imply that
all {the hypotheses} of Lemma~\ref{lem:pre_finishing} are satisfied and we can extend \(\chi\) into a complete acyclic edge colouring
of \(G\), which finishes the proof of Theorem~\ref{thm:main}.
\vspace{-0.5cm}
\bibliography{citation}{}
\bibliographystyle{plain}

\end{document}